\documentclass[11pt,reqno]{amsart}
\usepackage[osf,sc]{mathpazo}
\usepackage[utf8]{inputenc}
\usepackage{amsmath,amssymb,amsthm, mathrsfs}
\usepackage{indentfirst}
\usepackage{comment,relsize,braket}
\usepackage{enumerate,multicol}
\usepackage[numbers]{natbib}
\usepackage{graphicx}
\usepackage{color}
\usepackage[all]{xy}
\usepackage{fullpage}
\usepackage{tikz}
\usepackage[pagebackref,colorlinks, linkcolor=blue]{hyperref}
\usepackage[top=1in, bottom=1.3in, left=.7in, right=.7in]{geometry}
\newcommand{\floor[1]}{\lfloor #1 \rfloor}
\usepackage{fancyhdr}
\newtheorem{Theorem}{Theorem}[section]
\newtheorem{Lemma}[Theorem]{Lemma}
\newtheorem{Corollary}[Theorem]{Corollary}

\newtheorem{Remarks}[Theorem]{Remarks}
\newtheorem{Example}[Theorem]{Example}

\newtheorem{Definition}[Theorem]{Definition}

\numberwithin{equation}{section}

\everymath{\displaystyle}

\def\QQ{{\mathbb Q}} \def\NN{{\mathbb N}} \def\ZZ{{\mathbb Z}}
 \def\CC{{\mathbb C}}

\def\frk{\mathfrak}  
   
  \def\mm{{\frk m}}
\def\R{{\mathcal R}}  
\def\opn#1#2{\def#1{\operatorname{#2}}} 
\opn\chara{char} \opn\length{\ell} \opn\pd{pd} \opn\rk{rk}
\opn\projdim{proj\,dim} \opn\injdim{inj\,dim}
\opn\rank{rank} \opn\depth{depth} \opn\grade{grade} 
\opn\hei{ht} \opn\embdim{emb\,dim}\opn\codim{codim}
\opn\Tr{Tr} \opn\bigrank{big\,rank}
\opn\superheight{superheight} \opn\lcm{lcm}
\opn\rdim{rdim} \opn\trdeg{tr\,deg} \opn\reg{reg}  \opn\lreg{lreg} 
\opn\ini{in} \opn\lpd{lpd} \opn\size{size} \opn{\mult}{mult}
\opn\div{div} \opn\Div{Div} \opn\cl{cl} \opn\Cl{Cl}
\opn\Spec{Spec} \opn\Supp{Supp} \opn\supp{supp} 
\opn\Sing{Sing} \opn\Ass{Ass} \opn\Min{Min}
\opn\Proj{Proj} \opn{\Max}{Max} \opn{\Assh}{Assh}
\opn\Ann{Ann} \opn\Rad{Rad} \opn\Soc{Soc}
\opn\Syz{Syz} \opn\Im{Im} \opn\Ker{Ker} \opn\Coker{Coker}
\opn\Am{Am} \opn\Hom{Hom} \opn\tor{Tor} \opn\Ext{Ext}
\opn\End{End} \opn\Aut{Aut} \opn\id{id}

\opn\nat{nat} \opn\pff{pf} 
\opn\Pf{Pf} \opn\GL{GL} \opn\SL{SL} \opn\mod{mod} \opn\ord{ord}
\opn\Gin{Gin} \opn\Hilb{Hilb}
\opn\adeg{adeg} \opn\std{std}\opn\ip{infpt}
\opn\Pol{Pol} \opn\sat{sat} \opn\Var{Var}
\opn\aff{aff} \opn\con{conv} \opn\relint{relint} \opn\st{st}
\opn\lk{lk} \opn\cn{cn} \opn\core{core} \opn\vol{vol}
\opn\link{link} \opn\star{star}
\opn\gr{gr}

\opn\inn{in}

\title[]{Hilbert-Kunz function and Hilbert-Kunz multiplicity \\[2mm] of some ideals of the Rees algebra}
\thanks{The first author is supported by a UGC fellowship, Govt. of India}

\author[]{Kriti Goel}
\address{Indian Institute of Technology Bombay, Mumbai, India 400076}
\email{kritigoel.maths@gmail.com}

\author[]{Mitra Koley}
\address{School of Mathematics, Tata Institute of Fundamental Research, Homi Bhabha Road, Mumbai, India 400005}
\email{mitrak@math.tifr.res.in}

\author[]{J. K. Verma}
\address{Indian Institute of Technology Bombay, Mumbai, India 400076}
\email{jkv@iitb.ac.in}

\begin{document}

\begin{abstract}
	We prove that the Hilbert-Kunz function of the ideal $(I,It)$ of the Rees algebra $\R(I)$, where $I$ is an $\mm$-primary ideal of a $1$-dimensional local ring $(R,\mm)$, is a quasi-polynomial in $e$, for large $e.$ For $s \in \NN$, we calculate the Hilbert-Samuel function of the $R$-module $I^{[s]}$ and obtain an explicit description of the generalized Hilbert-Kunz function of the ideal $(I,It)\R(I)$ when $I$ is a parameter ideal in a Cohen-Macaulay local ring of dimension $d \geq 2$, proving that the generalized Hilbert-Kunz function is a piecewise polynomial in this case. 
\end{abstract}

\maketitle

\section{Introduction}

Let $(R,\mm)$ be a $d$-dimensional Noetherian local ring  of positive prime characteristic $p$ and let  $I$ be an $\mm$-primary ideal. The $q^{th}$-Frobenius power of $I$ is the ideal $I^{[q]}=(x^q\mid x\in I)$ where $q=p^e$ for $e\in \NN.$ The function $e \mapsto \ell_R(R/I^{[p^e]})$ is called the {\it Hilbert-Kunz function} of $R$ with respect to $I$ and was first considered by E. Kunz in \cite{kunz69}. In \cite{monsky}, P. Monsky showed that this function is of the form 
\[ \ell_R(R/I^{[q]}) = e_{HK}(I,R) q^{d} + O(q^{d-1}), \]
where $e_{HK}(I,R)$ is a positive real number called the {\it Hilbert–Kunz multiplicity} of $R$ with respect to $I.$ We write $e_{HK}(R) := e_{HK}(\mm,R)$ and $e_{HK}(I) := e_{HK}(I,R).$ Besides the mysterious leading coefficient, the behavior of the Hilbert–Kunz function is also unpredictable. Monsky proved that in the case of 1-dimensional rings, the Hilbert-Kunz function $\ell_R(R/I^{[p^e]}) = e_{HK}(I,R)q+\delta_e$, where $\delta_e$ is a periodic function of $e$, for large $e.$ Precisely, take $R=k[[X,Y]]/(X^5-Y^5)$, where $k$ is a field of prime characteristic $p\equiv \pm 2 (\mod 5).$ Let $\mm$ be the maximal ideal of $R.$  Monsky \cite{monsky} showed that for large $e$, $\ell_R(R/\mm^{[q]}) = 5q+\alpha_{\mm}(e)$, where $\alpha_{\mm}(e)=-4$ when $e$ is even and $\alpha_{\mm}(e)=-6$ when $e$ is odd. In \cite{hanMonsky}, the authors determined the Hilbert-Kunz function of the maximal ideal of the ring $R=\ZZ/p[[x_1,\ldots,x_s]]/(x_1^{d_1} + \cdots + x_s^{d_s})$, where $d_i$ are positive integers, $d=\dim R,$ and $s=d+1 \geq 3.$ They proved that if $\ell_R(R/\mm^{[q]}) = e_{HK}(R)q^d+\delta_e$, then $e_{HK}(R)$ is rational and $\delta_e$ is an eventually periodic function of $e$ whenever $s=3$ or $p=2.$ But when $s=4$ and $p=5$ with $d_1=\cdots=d_s=4$, then $\ell_R(R/\mm^{[5^e]}) = (168/61) \ 5^{3e} - (107/61) \ 3^e.$ In \cite{brenner2007}, H. Brenner proved the following result

\begin{Theorem}
	Let $k$ denote the algebraic closure of a finite field of characteristic $p.$ Let $R$ denote a normal two-dimensional standard-graded $k$-domain and let $I$ denote a homogeneous $R_+$-primary ideal. Then the Hilbert-Kunz function of $I$ has the form $\ell_R(R/I^{[q]}) = e_{HK}(I,R) q^2 + \delta_e$, where $e_{HK}(I,R)$ is a rational number and $\delta_e$ is an
	eventually periodic function.
\end{Theorem}

 There are numerous papers in which  Hilbert-Kunz function and multiplicity have been calculated for specific local and graded rings. We mention a few of them. K. Pardue \cite{p}, P. Monsky \cite{m2} and R.-O. Buchweitz and Q. Chen \cite{bc} computed the Hilbert-Kunz multiplicity for  cubic curves. N. Fakhruddin and V. Trivedi  \cite{tr1} found explicit formulas for the Hilbert-Kunz functions and Hilbert-Kunz multiplicities of certain projective embeddings of flag varieties and elliptic curves over algebraically closed fields of positive characteristics. Explicit computation of the Hilbert-Kunz function and multiplicity are also done in \cite{tr2}.
 In  \cite{tr3}, a family of one-dimensional reduced Cohen-Macaulay local rings are constructed which have same Hilbert-Kunz multiplicity but have infinitely many Hilbert-Kunz functions.

We recall a few definitions. Let $(R,\mm)$ be a Noetherian local ring of dimension $d$ and $I$ be an $\mm$-primary ideal of $R$. The {\it Hilbert-Samuel function} $H_I(n)$ of $R$ with respect to $I$ is defined as $H_I(n) = \ell_R(R/I^n).$ It is known that $H_I(n)$ is a polynomial function of $n$ of degree $d$, for large $n.$ In particular, there exists a polynomial $P_I(x) \in \QQ[x]$ such that $H_I(n) = P_I(n)$ for all large $n.$ Write 
\[ P_I(x) = e_0(I) \binom{x+d-1}{d} - e_1(I) \binom{x+d-2}{d-1} + \cdots + (-1)^d e_d(I), \]
where $e_i(I)$ for $i = 0, 1, \ldots, d$ are integers, called the {\it Hilbert coefficients of $I.$} The leading coefficient $e_0(I)$ is called the \emph{multiplicity} of $I$ and $e_1(I)$ is called the \emph{Chern number} of $I.$ We write $e_0(I) := e_0(I,R)$ and $e_0(R) := e_0(\mm,R).$ The {\it postulation number of $I$} is defined as 
\[ n(I) = \max\{n \mid H_I(n) \neq P_I(n)\}.\] 
The notion of reduction of an ideal was introduced by D. G. Northcott and D. Rees. Let $J \subseteq I$ be ideals of $R.$ If $JI^n = I^{n+1}$ for some $n$, then $J$ is called a {\it reduction of $I.$} The reduction number $r_J(I)$ of $I$ with respect to $J$ is the smallest $n$ such that $JI^n = I^{n+1}.$ The {\it reduction number} of $I$ is defined as
\[ r(I) = \min \{ r_J(I) \mid J \text{ is a minimal reduction of } I\}. \] 
Moreover, if $R/\mm$ is infinite, then minimal reduction of $I$ exists. 

Let $\R(I) = \bigoplus_{n \geq 0} I^nt^n$ denote the Rees algebra of $I$. The Hilbert-Kunz multiplicity of  various blowup  algebras was estimated by K. Eto and K.-i. Yoshida in their paper \cite{etoYoshida}. Put $c(d) = (d/2)+ d/(d+1)!.$ They proved the following. 

\begin{Theorem}
	Let $(R,\mm)$ be a Noetherian local ring of prime characteristic $p > 0$ with $d =\dim R \geq 1.$ Then for any $\mm$-primary ideal $I$, we have
	\[ e_{HK}(\R(I)) \leq c(d) \ e_0(I). \]
	Moreover, equality holds if and only if $e_{HK}(R)=e_0(I).$ When this is the case, $e_{HK}(R)=e_0(R)$ and $e_{HK}(I)=e_0(I).$
\end{Theorem}

In this paper, we compute the Hilbert-Kunz function and the Hilbert-Kunz multiplicity of certain ideals of the Rees algebra. We begin with the case of one-dimensional Noetherian local rings of prime characteristic, where, motivated by the result of Monsky, we prove the quasi-polynomial nature of Hilbert-Kunz function for certain classes of ideals of the Rees algebra. For rings of dimension at least 2, we prove that the generalized Hilbert-Kunz function is a piece-wise polynomial for certain class of ideals. The paper is organized as follows.
In section 2, we begin by proving the following result. 

\begin{Theorem}
	Let $(R,\mm)$ be a 1-dimensional Noetherian local ring with prime characteristic $p>0.$ Let $I,J$ be $\mm$-primary ideals. Then $e_{HK}((J,It)\R(I))=e_0(J).$
\end{Theorem}

Next, we prove that the Hilbert-Kunz function of the ideal $(I,It)$ of the Rees algebra $\R(I)$, where $I$ is an $\mm$-primary ideal of a 1-dimensional ring, is a quasi-polynomial in $e$, for large $e.$ Recall that a quasi-polynomial of degree $d$ is a function $f : \ZZ \rightarrow \CC$ of the form 
\[ f(n) = c_d(n)n^d + c_{d-1}(n)n^{d-1} + \cdots + c_0(n), \] 
where each $c_i(n)$ is a periodic function and $c_d(n)$ is not identically zero. Equivalently, $f$ is a quasi-polynomial if there exists an integer $N>0$ (namely, a common period of $c_0, c_1, \ldots, c_d$) and polynomials $f_0, f_1, \ldots, f_{N-1}$ such that
$f(n) = f_i(n)$ if $n \equiv i (\mod N)$ (see \cite{stanley97}).

\begin{Theorem}
	Let $R$ be a 1-dimensional Noetherian local ring with prime characteristic $p>0.$ Let $I$ be an $\mm$-primary ideal. Let $r$ be the reduction number of $I$ and $\rho$ be the postulation number of $I.$ Put $\mathcal{I}=(I,It)\R(I).$ Let $q=p^e$, where $e \in \NN$ is large. 
	\begin{enumerate}[{\rm (1)}]
		\item If $\rho+1 \leq r$, then
		\[ \ell_R\left(\frac{\R(I)}{\mathcal{I}^{[q]}}\right) = e_0(I) q^2 - e_0(I) \binom{r}{2} + r \ e_1(I) +  \sum_{n=0}^{r-1} \ell_R\left(\frac{R}{I^n}\right) + 2 \sum_{n=0}^{r-1} \alpha_I(I^n,e). \]
		
		\item If $r < \rho+1$, then
		\[ \ell_R\left(\frac{\R(I)}{\mathcal{I}^{[q]}}\right) = e_0(I) q^2 - e_0(I) \left( r(r-1) - \frac{\rho(\rho+1)}{2} \right) + (2r-\rho-1)e_1(I) + \beta + 2 \sum_{n=0}^{r-1} \alpha_I(I^n,e), \] 
	\end{enumerate}
	where $\beta = \sum_{n=0}^{r-1} \ell_R\left(\frac{R}{I^n}\right) - \sum_{n=r}^{\rho} \ell_R\left(\frac{R}{I^n}\right)$ is a constant and $\alpha_I(I^n,e) = \ell_R(I^n/I^{[q]}I^n) - e_0(I)q$ is a periodic function in $e$, for large $e.$
	In other words, $\ell_R(\R(I)/\mathcal{I}^{[q]})$ is a quasi-polynomial for large $e.$
\end{Theorem}

In the above setup, if $I$ is a parameter ideal and $J$ is an $\mm$-primary ideal of a Cohen-Macaulay local ring $R$, then we prove that for large $e$, 
\[ \ell_R \left(\frac{\R(I)}{(J,It)^{[q]}}\right) = q^2e_0(J) + q \ \alpha_J(e),\] 
where $\alpha_J(e) = \ell_R (R/J^{[q]}) - e_0(J)q$ is a periodic function of $e.$ Let $R=k[[X,Y]]/(X^5-Y^5)$, where $k$ is a field of prime characteristic $p \equiv \pm 2 (\mod 5)$ and $\mm$ is the maximal ideal of $R.$ As a consequence of the above results, we compute the Hilbert-Kunz function of the ideals $(\mm,\mm t)\R(\mm)$ and $(\mm,It)\R(I)$, where $I=(x)$ is a parameter ideal and $x$ is the image of $X$ in $R.$

Let $R$ be a $d$-dimensional Noetherian ring. Let $I$ be an ideal of finite co-length. Aldo Conca introduced the concept of generalized Hilbert-Kunz function in \cite{conca}. For $s \in \mathbb{N}$, let $I^{[s]} = (a_1^s, \dots, a_n^s)$ where $\{a_1,a_2,\dots,a_n\}$  is a fixed set of generators of $I.$ Then the {\it generalized Hilbert-Kunz function} is defined as
\[ HK_{R,I}(s) = \ell_R\left(\frac{R}{I^{[s]}}\right). \]
The generalized Hilbert-Kunz multiplicity is defined as $\lim\limits_{s \rightarrow \infty} HK_{R,I}(s)/s^d$ whenever the limit exists. If $\text{char} (R) = p > 0$ and $q$ is a power of $p,$ then the generalized Hilbert-Kunz function (multiplicity) coincides with the Hilbert-Kunz function (multiplicity) and is independent of the choice of the generators of $I.$ 

In section 3, we find the generalized Hilbert-Kunz function of the ideal $\mathcal {I}=(I,It)$ in $\R(I)$, when $I$ is generated by a regular sequence in a $d$-dimensional Cohen-Macaulay local ring.  Our approach requires knowledge of the Hilbert-Samuel function of the $R$-module $I^{[s]}.$ We obtain an explicit description of the function $F(s,n)=\ell_R (I^{[s]}/I^{[s]}I^n)$ for a fixed $s \in \mathbb{N}$ and then use some properties of Stirling numbers of the first kind and Stirling numbers of the second kind to prove the following result.
\begin{Theorem}
	Let $R$ be a $d$-dimensional Cohen-Macaulay local ring and $d \geq 2.$ Let $I$ be a parameter ideal of $R$ and $\mathcal I=(I,It)\R(I).$ 
	{Put $c(d) = (d/2)+ d/(d+1)!.$} 
	Let $s \in \mathbb{N}.$  \\
	{\rm(1)} Let $s<d.$ Write $d=k_1s+k_2$  where $k_2 \in \{0,1,\ldots,s-1\}.$ If $k_2=0$, then
	\[ \ell_R \left(\frac{\mathcal{R}(I)}{\mathcal I^{[s]}}\right) = 
	e_0(I) \left[ (d-k_1+1)s^{d+1} + d \binom{s+d-1}{d+1} - \sum_{i=0}^{d-1} \left[ (-1)^{i} \binom{d}{i} \binom{(d-i-k_1+1)s+d-1}{d+1} \right] \right]. \]
	If $k_2 \neq 0,$ then
	\[ \ell_R \left(\frac{\mathcal{R}(I)}{\mathcal I^{[s]}}\right) = 
	e_0(I) \left[ (d-k_1)s^{d+1} + d \binom{s+d-1}{d+1} - \sum_{i=0}^{d-1} \left[ (-1)^{i} \binom{d}{i} \binom{(d-i-k_1)s+d-1}{d+1} \right] \right]. \]
	{\rm(2)} Let $s \geq d.$ Then
	\[ \ell_R \left(\frac{\mathcal{R}(I)}{\mathcal I^{[s]}}\right) =  
	e_0(I) \left[ \frac{ds^{d+1}}{2} - \frac{s^d(d-2)}{2} + d \binom{s+d-1}{d+1} \right]. \]  
	In other words, for $s$ large, 
	\[ \ell_R \left(\frac{\R(I)}{\mathcal I^{[s]}}\right) = c(d)e_0(I) s^{d+1} + e_0(I)\left(\frac{d-2}{2}\right)\left(\frac{1}{(d-1)!} - 1 \right)s^d + e_0(I) \frac{d(d-1)(3d-10)}{24(d-1)!} s^{d-1} + \cdots , \]
	implying that the generalized Hilbert-Kunz multiplicity $e_{HK}((I,It)\R(I)) = c(d) \ e_0(I).$
\end{Theorem}

As a consequence, we obtain the following result.

\begin{Corollary}
	Let $(R,\mm)$ be a $d$-dimensional regular local ring with $d\geq 2.$ Then for $s \geq d$,
	\[ \ell_R \left(\frac{\R(\mm)}{(\mm,\mm t)^{[s]}}\right) = \frac{ds^{d+1}}{2} - \frac{s^d(d-2)}{2} + d \binom{s+d-1}{d+1}. \]
\end{Corollary} 

{In \cite{bgv}, we continue our exploration of the nature of Hilbert-Kunz function, where we give an expression of the generalized Hilbert-Kunz function of the maximal ideal of the Rees algebra of the maximal ideal of Stanley-Reisner rings.
A precise formula for the Hilbert-Samuel multiplicity of ideals of the Rees algebra was computed, in terms of mixed multiplicities, in \cite{jkvRees}. It is natural to ask if there is a parallel notion of mixed Hilbert-Kunz multiplicities and if the Hilbert-Kunz multiplicity can be expressed in terms of these numbers. This is still an open question.}

{\bf Acknowledgements}: We thank Vijaylaxmi Trivedi, Anurag Singh  and K.- i. Watanabe for several discussions and their lectures at IIT Bombay on Hilbert-Kunz multiplicity and positive characteristic methods.

\section{The Hilbert-Kunz function in dimension 1}

Let $(R,\mm)$ be a $1$-dimensional Noetherian local ring with prime characteristic $p>0$. In this section, we calculate the Hilbert-Kunz function and Hilbert-Kunz multiplicity of certain ideals of the Rees ring $\R(I)$, where $I$ is an $\mm$-primary ideal of $R.$

Let $I,J$ be $\mm$-primary ideals of $R.$ We begin by calculating $e_{HK}((J,It)\R(I))$, the Hilbert-Kunz multiplicity of the ideal $(J,It)$ in $\R(I).$ Recall that if a Noetherian ring $R$ has prime characteristic $p$, then $x\in R$ is said to be in the {\it tight closure} $I^*$ of an ideal $I$ if there exists $c\in R^o:=\{a\in R\mid a \notin \mathfrak p \text{ for any minimal prime } \mathfrak p \subset R\}$ such that $cx^q\in I^{[q]}$ for all large $q=p^e.$ 
In \cite[Theorem 8.17(a)]{hochsterHunekeTight}, M. Hochster and C. Huneke proved that if $J \subseteq I$ are $\mm$-primary ideals of $R$ such that $I^*=J^*$,  then $e_{HK}(I)=e_{HK}(J).$ 

For an ideal $I$ in a ring $R$, the {\it integral closure of $I$}, denoted by $\overline{I}$, is the ideal which consists of elements $x \in R$ such that $x$ satisfies an equation of the form $x^n + a_1x^{n-1} + \cdots + a_n = 0$ for some $a_i \in I^i$, $1 \leq i \leq n.$ The following generalization of the Brian\c{c}on-Skoda theorem was given by Hochster and Huneke.
\begin{Theorem} [{\cite[Theorem 5.4]{hochsterHunekeTight}}]
	Let $R$ be a Noetherian ring of prime characteristic $p$ and let $I$ be an ideal of positive height generated by $n$ elements. Then for every $m \in \NN$, $\overline{I^{n+m}} \subseteq (I^{m+1})^*.$
\end{Theorem}
It now follows that the tight closure and integral closure of principal ideals coincide.

\begin{Theorem} \label{dim1IJ}
	Let $(R,\mm)$ be a 1-dimensional Noetherian local ring with prime characteristic $p>0.$ Let $I,J$ be $\mm$-primary ideals. Then $e_{HK}((J,It)\R(I))=e_0(J).$
\end{Theorem}

\begin{proof}
	{If $R$ does not have an infinite residue field, we pass to a general extension $S=R[x]_{\mm[x]}$ of $R$, where $x$ is an indeterminate. Note that $S$ has an infinite residue field. Since $S/R$ is a faithfully flat extension, it follows that for all $q=p^e$, $e \in \NN$,
	\[ \ell\left( \frac{\R(I)}{(J,It)^{[q]}\R(I)} \right) 
	= \ell\left( \frac{\R(I)\otimes_R S}{(J,It)^{[q]}(\R(I) \otimes_R S)} \right) 
	= \ell\left( \frac{\R(I \otimes_R S)}{(J,It)^{[q]}\R(I \otimes_R S)} \right). \]
	Hence, $e_{HK}((J,It)\R(I)) = e_{HK}((J,It)\R(I \otimes_R S)).$
	}
	
	Therefore, we may assume that $R$ has an infinite residue field. Let $(x)$, $(y)$ be  minimal reductions of $J$ and $I$ respectively in $R.$ We claim that $e_{HK}((J,It)\R(I))=e_{HK}((x,yt)\R(I)).$  Note that it is sufficient to show that $((J,It)\R(I))^*=((x,yt)\R(I))^*.$ Using generalized Brian\c{c}on-Skoda theorem, it follows that $((x)\R(I))^* = \overline{(x)\R(I)}$ and $((yt)\R(I))^* = \overline{(yt)\R(I)}.$ As $(x)\R(I)$ and  $(yt)\R(I)$ are reductions of $J\R(I)$ and $(It)\R(I)$ respectively, we have
	\begin{align*}
		((x)\R(I))^* = \overline{(x)\R(I)} = \overline{J\R(I)} \ \text{ and } \  ((yt)\R(I))^* = \overline{(yt)\R(I)}=\overline{(It)\R(I)}.
	\end{align*}
	Since $\overline{J\R(I)} = ((x)\R(I))^* \subseteq (J\R(I))^*$, we get $((x)\R(I))^* = (J\R(I))^*.$ Similarly, $((yt)\R(I))^* = ((It)\R(I))^*.$ Consider
	\begin{align*}
		((x,yt)\R(I))^* = \big( ((x)\R(I))^* + ((yt)\R(I))^* \big)^* = \big( (J\R(I))^* + ((It)\R(I))^* \big)^* = ((J,It)\R(I))^*.
	\end{align*}
	This proves the claim. Therefore, $e_{HK}((J,It)\R(I))=e_{HK}((x,yt)\R(I)).$ As $x,yt$ form a system of parameters in $\R(I)_{(\mm,It)}$, it follows that 
	\begin{align*}
		e_{HK}((J,It)\R(I)) = e_{HK}((x,yt)\R(I)) = e_0((x,yt)\R(I)) = e_0((J,It)\R(I)) = e_0(J), 
	\end{align*}
	where the last equality follows from \cite[Theorem 3.1]{jkvRees}.
\end{proof}

Let $I$ be an $\mm$-primary ideal of $R.$ We prove that $\ell_R(\R(I)/(I,It)^{[p^e]})$ is a quasi-polynomial in $e$, for large $e.$

\begin{Theorem} \label{dim1}
	Let $(R,\mm)$ be a 1-dimensional Noetherian local ring with prime characteristic $p>0.$ Let $I$ be an $\mm$-primary ideal. Let $r$ be the reduction number of $I$ and $\rho$ be the postulation number of $I.$ Put $\mathcal{I}=(I,It)\R(I).$ Let $q=p^e$, where $e \in \NN$ is large. 
	\begin{enumerate}[{\rm (1)}]
		\item If $\rho+1 \leq r$, then
		\[ \ell_R\left(\frac{\R(I)}{\mathcal{I}^{[q]}}\right) = e_0(I) q^2 - e_0(I) \binom{r}{2} + e_1(I) r +  \sum_{n=0}^{r-1} \ell_R \left(\frac{R}{I^n}\right) + 2 \sum_{n=0}^{r-1} \alpha_I(I^n,e). \]

		\item If $r < \rho+1$, then
		\[ \ell_R\left(\frac{\R(I)}{\mathcal{I}^{[q]}}\right) = e_0(I) q^2 - e_0(I) \left( r(r-1) - \frac{\rho(\rho+1)}{2} \right) + (2r-\rho-1)e_1(I) + \beta + 2 \sum_{n=0}^{r-1} \alpha_I(I^n,e), \] 
	\end{enumerate}
	where $\beta = \sum_{n=0}^{r-1} \ell_R\left(\frac{R}{I^n}\right) - \sum_{n=r}^{\rho} \ell_R\left(\frac{R}{I^n}\right)$ is a constant and $\alpha_I(I^n,e) = \ell_R(I^n/I^{[q]}I^n) - e_0(I)q$ is a periodic function in $e$, for large $e.$
	In other words, $\ell_R(\R(I)/\mathcal{I}^{[q]})$ is a quasi-polynomial for large $e.$
\end{Theorem}

\begin{proof}
	{If $R$ does not have an infinite residue field, we pass to a general extension $S=R[x]_{\mm[x]}$ of $R$, where $x$ is an indeterminate. Note that $S$ has an infinite residue field. Since $S/R$ is a faithfully flat extension, it follows that for all $q=p^e$, $e \in \NN$,
	\[ \ell\left( \frac{\R(I)}{(I,It)^{[q]}\R(I)} \right) 
	= \ell\left( \frac{\R(I)\otimes_R S}{(I,It)^{[q]}(\R(I) \otimes_R S)} \right) 
	= \ell\left( \frac{\R(I \otimes_R S)}{(I,It)^{[q]}\R(I \otimes_R S)} \right). \]
	}
	
	Therefore, we may assume that $R$ has an infinite residue field. Fix $q=p^e$ large. Observe that 
	\begin{align*}
	\mathcal{I}^{[q]} 
	= (I^{[q]}, I^{[q]}t^q) 
	= \left(\bigoplus_{n=0}^{q-1} I^{[q]}I^nt^n \right) + \left(\bigoplus_{n \geq q} I^{[q]}I^{n-q} t^n \right). 
	\end{align*}
	Let $(x)$ be a minimal reduction of $I$ and let $r$ be the reduction number of $I.$ Then $x^kI^l = I^{k+l}$ for all $k \geq 1$ and $l \geq r.$ Write $I=x+J$, for some ideal $J \subseteq I.$ Then $I^{[q]}I^{n-q} = (x^q + J^{[q]})I^{n-q} = I^n$, for all $n \geq q+r.$ Therefore,
	\begin{align*}
	\mathcal{I}^{[q]} 
	= \left(\bigoplus_{n=0}^{q-1} I^{[q]}I^nt^n \right) + \left(\bigoplus_{n=q}^{q+r-1}  I^{[q]}I^{n-q}t^n \right) + \left(\bigoplus_{n \geq q+r}I^nt^n \right)
	\end{align*}
	and hence
	\begin{align*}
	\ell_R \left(\frac{\mathcal{R}(I)}{\mathcal{I}^{[q]}}\right)
	&= \sum_{n=0}^{r-1} \ell_R \left(\frac{I^n}{I^{[q]}I^n}\right) 
	+ \sum_{n=r}^{q-1} \ \ell_R \left(\frac{I^n}{I^{[q]}I^n}\right)
	+ \sum_{n=q}^{q+r-1} \ell_R \left(\frac{I^n}{I^{[q]}I^{n-q}} \right).
	\end{align*}
	{Writing $\ell_R(I^n/I^{[q]}I^n) = \ell_R(R/I^{[q]}I^n) - \ell_R(R/I^n)$, we get}
	\begin{align*}
	{\ell_R \left(\frac{\mathcal{R}(I)}{\mathcal{I}^{[q]}}\right)}
	&{= \sum_{n=0}^{r-1} \ell_R \left(\frac{I^n}{I^{[q]}I^n}\right) + \sum_{n=r}^{q-1} \ell_R \left(\frac{R}{I^{[q]}I^n}\right) 
	+ \sum_{n=q}^{q+r-1} \ell_R \left(\frac{R}{I^{[q]}I^{n-q}}\right) 
	- \sum_{n=r}^{q+r-1} \ell_R \left(\frac{R}{I^n}\right)} \\
	&{= \sum_{n=0}^{r-1} \ell_R \left(\frac{I^n}{I^{[q]}I^n}\right) + \sum_{n=r}^{q-1} \ell_R \left(\frac{R}{I^{[q]}I^n}\right) 
	+ \sum_{n=0}^{r-1} \ell_R \left(\frac{R}{I^{[q]}I^n}\right)
	- \sum_{n=r}^{q+r-1} \ell_R \left(\frac{R}{I^n}\right)} \\	
	&= 2\sum_{n=0}^{r-1} \ell_R \left(\frac{I^n}{I^{[q]}I^n}\right) + \sum_{n=r}^{q-1} \ell_R \left(\frac{R}{I^{[q]}I^n}\right) + \sum_{n=0}^{r-1} \ell_R \left(\frac{R}{I^n}\right) - \sum_{n=r}^{q+r-1} \ell_R \left(\frac{R}{I^n}\right).
	\end{align*}
	For $n \geq r$, $I^{[q]}I^n = (x^q + J^{[q]})I^n = I^{n+q}.$ If the given ring is not complete, we may pass to its $\mm$-adic completion $\hat{R}$ and use the fact that
	$\hat{R}$ is a faithfully flat extension of $R.$ Now use  \cite[Theorem 3.11]{monsky},  to  write $\ell_R(I^n/I^{[q]}I^n) = e_0(I,I^n)q + \alpha_I(I^n,e)$, where $\alpha_I(I^n,e)$ is a periodic function of $e.$ Using the associativity formula for $e_0(I)$, one may check that $e_0(I,I^n) = e_0(I,R).$ Thus,
	\begin{align*}
	\ell_R \left(\frac{\mathcal{R}(I)}{\mathcal{I}^{[q]}}\right)
	&= 2\sum_{n=0}^{r-1} \big( e_0(I)q + \alpha_I(I^n,e) \big) + \sum_{n=r}^{q-1} \ell_R \left(\frac{R}{I^{n+q}}\right) + \sum_{n=0}^{r-1} \ell_R \left(\frac{R}{I^n}\right) - \sum_{n=r}^{q+r-1} \ell_R \left(\frac{R}{I^n}\right). 
	\end{align*}
	{Since evaluation of the expression $\sum_{n=r}^{q+r-1} \ell_R \left(\frac{R}{I^n}\right)$ depends on the relation between $r$ and $\rho$, we have the following cases:}
	
	{\bf Case 1}: Let $\rho + 1 \leq r.$ Then
	{$\ell_R(R/I^n) = e_0(I)n - e_1(I)$, for all $n \geq r.$ Hence,}
	\begin{align*}
	\ell_R \left(\frac{\mathcal{R}(I)}{\mathcal{I}^{[q]}}\right)
	&= 2rq \ e_0(I) + 2 \sum_{n=0}^{r-1} \alpha_I(I^n,e) + \sum_{n=r+q}^{2q-1} \big( e_0(I)n-e_1(I) \big) + \sum_{n=0}^{r-1} \ell_R \left(\frac{R}{I^n}\right)- \sum_{n=r}^{q+r-1} \big( e_0(I)n-e_1(I) \big) \\
	&= 2rq \ e_0(I) + 2 \sum_{n=0}^{r-1} \alpha_I(I^n,e) + e_0(I) \left[\binom{2q}{2} - 2\binom{r+q}{2} + \binom{r}{2} \right]  + e_1(I) r + \sum_{n=0}^{r-1} \ell_R \left(\frac{R}{I^n}\right).
	\end{align*}
	{Here, we use the fact that for any $a,b \in \mathbb{N},$ $\sum_{a}^{b-1} n  = \binom{b}{2} - \binom{a}{2}$. Since $2rq + \binom{2q}{2} - 2\binom{r+q}{2} + \binom{r}{2} = q^2 - \binom{r}{2}$, we get}
	\begin{align*}
	\ell_R \left(\frac{\mathcal{R}(I)}{\mathcal{I}^{[q]}}\right)
	= e_0(I) q^2 - e_0(I) \binom{r}{2} + e_1(I) r +  \sum_{n=0}^{r-1} \ell_R \left(\frac{R}{I^n}\right) + 2 \sum_{n=0}^{r-1} \alpha_I(I^n,e).
	\end{align*}	
	
	{\bf Case 2}: Let $r < \rho+1.$ Then
	{$\ell_R(R/I^n) = e_0(I)n - e_1(I)$, for all $n \geq \rho+1$ and  $\sum_{n=r}^{\rho} \ell_R(R/I^n)$ is a constant. Hence,}
	\begin{align*}
	\ell_R \left(\frac{\mathcal{R}(I)}{\mathcal{I}^{[q]}}\right)
	&= 2rq \ e_0(I) + 2 \sum_{n=0}^{r-1} \alpha_I(I^n,e) + \sum_{n=r+q}^{2q-1} \big( e_0(I)n-e_1(I) \big) + \sum_{n=0}^{r-1} \ell_R\left(\frac{R}{I^n}\right) - \sum_{n=r}^{\rho} \ell_R \left(\frac{R}{I^n}\right) \\
	&\hspace{10cm}- \sum_{n=\rho+1}^{q+r-1} \big( e_0(I)n-e_1(I) \big).
	\end{align*}
	{Rearranging the terms, we get}
	\begin{align*}
	\ell_R \left(\frac{\mathcal{R}(I)}{\mathcal{I}^{[q]}}\right)
	&= 2rq \ e_0(I) + 2 \sum_{n=0}^{r-1} \alpha_I(I^n,e) + e_0(I) \left[\binom{2q}{2} - 2\binom{r+q}{2} + \binom{\rho+1}{2} \right]  + (2r-\rho-1)e_1(I) + \beta \\
	&= e_0(I) q^2 - e_0(I) \left( r(r-1) - \frac{\rho(\rho+1)}{2} \right) + (2r-\rho-1)e_1(I) + \beta + 2 \sum_{n=0}^{r-1} \alpha_I(I^n,e),
	\end{align*}
	where $\beta = \sum_{n=0}^{r-1} \ell_R\left(\frac{R}{I^n}\right) - \sum_{n=r}^{\rho} \ell_R\left(\frac{R}{I^n}\right)$ is a constant. Let $N$ be a common period of $\alpha_I(I^n,e)$, for $0 \leq n \leq r-1.$ Then it follows that $\ell_R(\R(I)/\mathcal{I}^{[q]}) - e_0(I)q^2$ is a periodic function of $e.$ In other words, $\ell_R(\R(I)/\mathcal{I}^{[q]})$ is a quasi-polynomial in $e$, for large $e.$
\end{proof}

In particular, if $R$ is a $1$-dimensional Cohen-Macaulay local ring we get the following result.

\begin{Corollary}  \label{cordim1}
	Let $(R,\mm)$ be a 1-dimensional Cohen-Macaulay local ring with prime characteristic $p>0.$ Let $I$ be an $\mm$-primary ideal. Let $r$ be the reduction number of $I.$ Put $\mathcal{I}=(I,It)\R(I).$ Then for $q=p^e$, where $e \in \NN$ is large,
	\[ \ell_R\left(\frac{\R(I)}{\mathcal{I}^{[q]}}\right) = e_0(I) q^2 - e_0(I) \binom{r}{2} + e_1(I) r + \sum_{n=0}^{r-1} \ell_R\left(\frac{R}{I^n}\right) + 2 \sum_{n=0}^{r-1} \alpha_I(I^n,e), \] 
	where $\alpha_I(I^n,e) = \ell_R(I^n/I^{[q]}I^n) - e_0(I)q$ is a periodic function in $e$, for large $e.$
	In other words, $\ell_R(\R(I)/\mathcal{I}^{[q]})$ is a quasi-polynomial for large $e.$
\end{Corollary}

\begin{proof}
	Since $R$ is Cohen-Macaulay, using \cite[Theorem 2.15]{marleyThesis} it follows that the postulation number of $I$, $\rho = r - 1.$ Substitute the same in Theorem \ref{dim1} to conclude.
\end{proof}

Next, we calculate the Hilbert-Kunz function of the ideal $\mathcal{J}=(J,It)\R(I)$ when $I$ is a parameter ideal and $J$ is an $\mm$-primary ideal of a Cohen-Macaulay ring $R.$

\begin{Theorem} \label{sopdim1}
	Let $(R,\mm)$ be a $1$-dimensional Cohen-Macaulay local ring with prime characteristic $p>0.$ Let $I=(a)$ be a parameter ideal and $J$ be an $\mm$-primary ideal. Put 
	$\mathcal{J}=(J,It)\R(I).$ Then for $q=p^e$, where $e \in \NN$ is large, 
	\[ \ell_R\left(\frac{\R(I)}{\mathcal{J}^{[q]}}\right) = q^2e_0(J) + q \ \alpha_J(e),\]
	where $\alpha_J(e) = \ell_R(R/J^{[q]}) - e_0(J)q$ is a periodic function of $e.$
\end{Theorem}

\begin{proof}
	Observe that 
	\[ \mathcal{J}^{[q]} = (J^{[q]},I^{[q]}t^q) 
	= \left(\bigoplus_{n=0}^{q-1} J^{[q]}(a^n)t^n \right) + \left(\bigoplus_{n \geq q} (a^n)t^n \right) \]
	which implies that
	\begin{align*}
	\ell_R\left(\frac{\mathcal{R}(I)}{\mathcal{J}^{[q]}}\right)
	= \sum_{n=0}^{q-1} \ell_R\left(\frac{(a^n)}{J^{[q]}(a^{n})}\right) 
	= q \ \ell_R \left(\frac{R}{J^{[q]}}\right) 
	=e_0(J)q^2 + q \ \alpha_J(e).
	\end{align*}
\end{proof}

We illustrate the above results in the following example.

\begin{Example}{\rm
	Let $R=k[[X,Y]]/(X^5-Y^5)$, where $k$ is a field of prime characteristic $p\equiv \pm 2 (\mod 5).$ Let $\mm$ be the maximal ideal of $R.$ Put $q=p^e$ for some $e \in \NN.$ Monsky \cite{monsky} showed that for large $e$, $\ell_R(R/\mm^{[q]}) = 5q+\alpha_{\mm}(e)$, where $\alpha_{\mm}(e)=-4$ when $e$ is even and $\alpha_{\mm}(e)=-6$ when $e$ is odd. 
	
	(1) We first calculate $\ell_R(\R(\mm)/(\mm,\mm t)^{[q]}).$ It is easy to check that $e_0(\mm)=5$ and $e_1(\mm)=10.$ Let $I=(x)$, where $x$ denotes the image of $X$ in $R.$ Then $I$ is a minimal reduction of $\mm$ and $r(\mm)=4.$ Using Corollary \ref{cordim1}, it follows that
	\begin{align} \label{eqex}
	\ell_R\left(\frac{\R(\mm)}{(\mm,\mm t)^{[q]}}\right) 
	= 5q^2 + 10 + \sum_{n=0}^{3} \ell_R\left(\frac{R}{\mm^n}\right) + 2 \sum_{n=0}^{3} \alpha_{\mm}(\mm^n,e) 
	= 5q^2 + 20 + 2 \sum_{n=0}^{3} \alpha_{\mm}(\mm^n,e).
	\end{align}
	{For an $R$-ideal $J$, let $\mu(J)$ denote the minimal number of generators of $J.$}
	For $n=1,2,3,$ writing 
	\begin{align*}
	\ell_R\left(\frac{\mm^n}{\mm^{[q]}\mm^n}\right) = \ell_R\left(\frac{\mm^{[q]}}{\mm^{[q]}\mm^n}\right) - \ell_R\left(\frac{R}{\mm^n}\right) + \ell_R\left(\frac{R}{\mm^{[q]}}\right)
	= \sum_{i=0}^{n-1} \mu(\mm^{[q]}\mm^i) - \ell_R\left(\frac{R}{\mm^n}\right) + \ell_R\left(\frac{R}{\mm^{[q]}}\right),
	\end{align*}
	we obtain
	\begin{center}
	\begin{minipage}{.30 \columnwidth}
		\begin{align*}
		\alpha_{\mm}(\mm,e) = 
		\begin{cases}
		-3 & \text{if } e \text{ is even} \\
		-5 & \text{if } e \text{ is odd},
		\end{cases}
		\end{align*}
	\end{minipage}
	\begin{minipage}{.40 \columnwidth}
		\begin{align*}
		\alpha_{\mm}(\mm^2,e) = 
		\begin{cases}
		-2 & \text{if } e \text{ is even} \\
		-3 & \text{if } e \text{ is odd},
		\end{cases}
		\end{align*}
	\end{minipage}
	\begin{minipage}{.20 \columnwidth}
		\begin{align*}
		\alpha_{\mm}(\mm^3,e) = -1.
		\end{align*}
	\end{minipage}
	\end{center}

	\vspace{5pt}
	Substituting in \eqref{eqex}, we get $\ell_R(\R(\mm)/(\mm,\mm t)^{[q]}) = 5q^2 + \beta_e$, where $\beta_e=0$ when $e$ is even and $\beta_e=-10$ when $e$ is odd.
	
	(2) We now calculate $\ell_R(\R(I)/(\mm,It)^{[q]}).$ Using Theorem \ref{sopdim1}, we get 
	\begin{align*} 
	\ell_R\left(\frac{\R(I)}{(\mm,It)^{[q]}}\right) 
	= e_0(\mm) q^2 + q \ \alpha_{\mm}(e) 
	= 5q^2 + 
	\begin{cases}
	-4q & \text{if } e \text{ is even}  \\
	-6q & \text{if } e \text{ is odd}.
	\end{cases}
	\end{align*}
	One can also verify this using the following arguments. Observe that $\R(I) \simeq k[[X,Y]][Z]/(X^5-Y^5).$ In order to find $\ell(k[X,Y,Z]/(X^5-Y^5,X^q,Y^q,Z^q))$, we find the Gr\"{o}bner basis of the ideal $M_q=(X^5-Y^5,X^q,Y^q,Z^q)$ in $k[X,Y,Z].$ Let `$>$' be any monomial ordering on $k[X,Y,Z]$ with $X>Y>Z.$ Since $q$ is large, the $S$-polynomials are:
	\begin{align*}
	S(X^5-Y^5,X^q) = X^{q-5}Y^5, && S(X^5-Y^5,X^{q-5}Y^5) = X^{q-10}Y^{10}, \ldots
	\end{align*}
	Therefore, Gr\"{o}bner basis of $M_q$ is $G = \{X^5-Y^5,X^q,Y^q,Z^q,X^{q-5i}Y^{5i} | q-5i>0\}.$ Observe that 
	\begin{align*}
	\ell_R\left(\frac{\R(I)}{(\mm,It)^{[q]}}\right) 
	= \ell\left(\frac{k[X,Y,Z]}{\inn_{>}(M_q)}\right) 
	= \ell\left(\frac{k[X,Y,Z]}{(X^5,Y^q,Z^q,X^{q-5i}Y^{5i} \mid q-5i<5)}\right).
	\end{align*}
	We now explore the condition $q-5i$ such that $q-5i<5.$ Since $q$ is of the form $(5k+2)^e$ or $(5k+3)^e$ for some $k,e \in \NN$, it is sufficient to find the values $(5k+2)^e$ modulo 5 and $(5k+3)^e$ modulo 5. First, consider the case $(5k+2)^e$ modulo 5. Using the binomial theorem, we only need to find $2^e (\mod 5).$ 
	
	If $e$ is even, then $2^e \equiv 1,4 (\mod 5).$ This implies that $X^{q-5i}Y^{5i} = XY^{q-1}$ or $X^{q-5i}Y^{5i} = X^4Y^{q-4}.$ In either of the case, we get $\ell_R(\R(I)/(\mm,It)^{[q]}) = 5q^2 - 4q.$ 
	
	If $e$ is odd, then $2^e \equiv 2,3 (\mod 5).$ This implies that $X^{q-5i}Y^{5i} = X^2Y^{q-2}$ or $X^{q-5i}Y^{5i} = X^3Y^{q-3}.$ In either of the case, $\ell_R(\R(I)/(\mm,It)^{[q]}) = 5q^2 - 6q.$ We get the same conclusion in the case $(5k+3)^e$ modulo $5.$
}\end{Example}

\section{The Hilbert-Kunz function in dimension $\geq 2.$}

In this section, we find the generalized Hilbert-Kunz function of the ideal $(I,It)\R(I)$, where $I$ is a parameter ideal in a Cohen-Macaulay local ring. It turns out that in this case, the generalized Hilbert-Kunz function is eventually a polynomial.

Let $R$ be a Cohen-Macaulay Stanley-Reisner ring of a simplicial complex over an infinite field with prime characteristic $p>0.$ Let $\mm$ be the maximal homogeneous ideal of $R$ and $I$ be an ideal generated by a linear system of parameters. It is proved in \cite[Theorem 6.1]{gmv} that $\mm=I^*.$ Therefore, using \cite[Corollary 4.5]{etoYoshida} we get $e_{HK}((\mm,\mm t)\R(\mm)) = e_{HK}((I,It)\R(I)).$ Thus in order to calculate $e_{HK}(\R(\mm)),$ it is sufficient to calculate $e_{HK}((I,It)\R(I)).$ This observation motivated us to consider the following setup.

Let $R$ be a Cohen-Macaulay local ring and $I$ be a parameter ideal. Let $G(I)=\oplus_{n \geq 0}I^n/I^{n+1}$ denote the associated graded ring of $I.$ Fix $s \in \mathbb{N}.$ For a fixed set of generators of $I$, define functions 
\[F(s,n) := H_I(I^{[s]},n) = \ell_R\left( \dfrac{I^{[s]}}{I^{[s]}I^n} \right) \text{ \ and \ } H(n) := H_I(R,n) = \ell_R\left( \dfrac{R}{I^n} \right) = e_0(I) \binom{n+d-1}{d} \] 
for all $n.$ Note that if $R$ is $1$-dimensional, then $F(s,n)=H(n)$ for all $n.$

We shall use the following result to find the reduction number of powers of an $\mm$-primary ideal. 

\begin{Theorem}[{\cite[Corollary 2.21]{marleyThesis}}] \label{Marley}
	Let $(R, \mm)$ be a $d$-dimensional Cohen-Macaulay local ring with infinite residue field and $I$ be an $\mm$-primary ideal such that $\grade(G(I)_+) \geq d - 1.$ Then for $k \geq 1$, 
	\[ r(I^k) = \floor[\frac{n(I)}{k}] + d. \]
\end{Theorem}

\begin{Theorem}  \label{F(n)}
	Let $R$ be a $d$-dimensional Cohen-Macaulay local ring and let $I$ be a parameter ideal. Let $d \geq 2.$ For a fixed $s \in \mathbb{N}$,
	\begin{align} \label{expressF(n)}
	F(s,n) = 
	\begin{cases}
	d \ H(n) &  \text{if } 1 \leq n \leq s, \\
	\sum_{i=1}^{d-1} (-1)^{i+1} \binom{d}{i} H(n-(i-1)s) & \text{ if }  s+1 \leq n \leq (d-1)s-1, \\
	H(n+s) - s^d e_0(I) & \text{ if } n \geq (d-1)s.
	\end{cases}
	\end{align}
	{In particular, the function $F(s,n)$ is independent of the choice of generating set of $I.$}
\end{Theorem}

\begin{proof} 
	{\bf Case 1}: Let $1 \leq n \leq s.$ Let $I=(x_1,\ldots,x_d).$ Consider the following epimorphism
	\begin{align*}
	\left(\frac{R}{I^n}\right)^{\oplus d} &\overset{\phi}\longrightarrow \frac{I^{[s]}}{I^{[s]}I^n} \rightarrow 0 \\
	(a_1+I^n,\ldots,a_d+I^n) &\mapsto (a_1x_1^s + \cdots + a_d x_d^s) + I^{[s]}I^n.
	\end{align*}
	Let $(a_1+I^n,\ldots,a_d+I^n) \in \text{ker}(\phi).$ Then $a_1x_1^s+\cdots+a_dx_d^s \in I^{[s]}I^n$ implies that 
	\[a_1x_1^s+\cdots+a_dx_d^s = b_1x_1^s+\cdots+b_dx_d^s, \] 
	where $b_1,\ldots,b_d \in I^n.$ Since $x_1^s,\ldots,x_d^s$ is an $R$-regular sequence, it follows that $(a_i - b_i) \in I^{[s]} \subseteq I^n$ for all $i =1,\ldots,d.$ Thus $a_1,\ldots,a_d \in I^n.$ Therefore, $\text{ker}(\phi)=0$ and $I^{[s]}/I^{[s]}I^n \simeq (R/I^n)^{\oplus d}.$ Hence $F(s,n) = d \ H(n)$, for all $1 \leq n \leq s.$
	
	{\bf Case 2}: Let $s+1 \leq n \leq (d-1)s-1.$ Let $x_i^*$ denote the image of $x_i$ in $I/I^2$, for all $i=1,\ldots,d.$ As $G(I)$ is Cohen-Macaulay, $(\textbf{x}^*)^{[s]} = (x_1^s)^*,\ldots,(x_d^s)^* \in I^s/I^{s+1}$ is  a $G(I)$-regular sequence. Hence  the following exact sequence is
	obtained from the Koszul complex of  $G(I)$ with respect to $(\textbf{x}^*)^{[s]} .$
	{\small \begin{align*}
	0 \rightarrow G(I)(-(d-1)s) \rightarrow G(I)(-(d-2)s)^{\binom{d}{1}} \rightarrow \cdots 
	\rightarrow G(I)^{\binom{d}{1}} \rightarrow G(I)(s) \rightarrow H_0((\textbf{x}^*)^{[s]};G(I)(s)) \rightarrow 0.
	\end{align*}}
	As we have an exact sequence of graded $G(I)$-modules, taking the $n^{th}$-graded component of each of these modules gives us the following exact sequence
	\begin{align*}
	0 \rightarrow \frac{I^{n-(d-1)s}}{I^{n-(d-1)s+1}} \rightarrow \left(\frac{I^{n-(d-2)s}}{I^{n-(d-2)s+1}}\right)^{\binom{d}{1}} \rightarrow \cdots \rightarrow \left(\frac{I^n}{I^{n+1}}\right)^{\binom{d}{1}} \rightarrow \frac{I^{n+s}}{I^{n+s+1}} \rightarrow \frac{I^{n+s}}{I^{[s]}I^{n}+I^{n+s+1}} \rightarrow 0.
	\end{align*}
	For $s \leq n \leq (d-1)s-1,$ it follows that
	\begin{align*}
	\ell_R\left(\frac{I^{n+s}}{I^{[s]}I^{n}+I^{n+s+1}}\right) = \ell_R\left(\frac{I^{n+s}}{I^{n+s+1}}\right) + \sum_{i=1}^{d-1} (-1)^i \binom{d}{i} \ell_R\left(\frac{I^{n-(i-1)s}}{I^{n-(i-1)s+1}}\right). 
	\end{align*}	
	{By canceling $\ell_R(R/I^{n+s})$ on both the sides, we get}
	\begin{align} \label{eq11}
	\ell_R\left(\frac{R}{I^{[s]}I^{n}+I^{n+s+1}}\right) = \ell_R\left(\frac{R}{I^{n+s+1}}\right) + \sum_{i=1}^{d-1} (-1)^i \binom{d}{i} \ell_R\left(\frac{I^{n-(i-1)s}}{I^{n-(i-1)s+1}}\right).
	\end{align}	
	We can also write
	\begin{align}
	\ell_R\left(\frac{R}{I^{[s]}I^{n}+I^{n+s+1}}\right) 
	&= \ell_R\left(\frac{R}{I^{[s]}}\right) + \ell_R\left(\frac{I^{[s]}}{I^{[s]}I^{n}}\right)  - \ell_R\left(\frac{I^{[s]}I^{n}+I^{n+s+1}}{I^{[s]}I^{n}}\right) \nonumber\\
	&= \ell_R\left(\frac{R}{I^{[s]}}\right) + \ell_R\left(\frac{I^{[s]}}{I^{[s]}I^{n}}\right)  - \ell_R\left(\frac{I^{n+s+1}}{I^{[s]}I^{n+1}}\right) \label{eq12}
	\end{align}
	since using Valabrega-Valla criterion (\cite[Corollary 2.7]{VV}),
	\begin{align*}
	\frac{I^{[s]}I^{n}+I^{n+s+1}}{I^{[s]}I^{n}} \simeq \frac{I^{n+s+1}}{(I^{[s]} \cap I^{n+s}) \cap I^{n+s+1}} \simeq \frac{I^{n+s+1}}{I^{[s]} \cap I^{n+s+1}} \simeq \frac{I^{n+s+1}}{I^{[s]}I^{n+1}}.
	\end{align*}
	Combining equations \eqref{eq11} and \eqref{eq12}, we get
	\begin{align*}
	\ell_R\left(\frac{I^{[s]}}{I^{[s]}I^{n}}\right) - \ell_R\left(\frac{I^{[s]}}{I^{[s]}I^{n+1}}\right)
	= \sum_{i=1}^{d-1} (-1)^i \binom{d}{i} \ell_R\left(\frac{I^{n-(i-1)s}}{I^{n-(i-1)s+1}}\right) 
	\end{align*}
	and hence for all $s \leq n \leq (d-1)s-1,$
	\begin{align*}
	F(s,n+1) - F(s,n) = \sum_{i=1}^{d-1} (-1)^{i+1} \binom{d}{i} \big[ H(n-(i-1)s+1) - H(n-(i-1)s) \big].
	\end{align*}
	Adding the above equality from $s$ to $t-1$, for any $t$ such that $s+1 \leq t \leq (d-1)s-1$, we get
	\begin{align*}
	F(s,t) - F(s,s) 
	&= \sum_{i=1}^{d-1} (-1)^{i+1} \binom{d}{i} [ H(t-(i-1)s) - H((2-i)s) ] \\
	&= d[H(t) - H(s) ] + \sum_{i=2}^{d-1} (-1)^{i+1} \binom{d}{i} H(t-(i-1)s).
	\end{align*}
	Since $F(s,s) = d \ H(s),$ we get the result for $F(s,n)$, $s+1 \leq n \leq (d-1)s-1.$
	
	{\bf Case 3}: Let $n \geq s(d-1).$ 
%
{	
	We prove that $I^{[s]}I^n = I^{n+s}$. Let $I=(x_1,...,x_d)$, then $I^{n+s}$ is generated by the monomials of the form $x_1^{i_1} \cdots x_d^{i_d}$ where $\sum i_j = n+s$. Since the monomial $x_1^{s-1} \cdots x_d^{s-1}$ has degree $ds-d$ and $n+s\geq ds$, it follows that $I^{n+s}\subseteq I^{[s]}I^n\subseteq I^{n+s}$. Thus for all $n \geq (d-1)s$, 
	\[ F(s,n) = \ell_R\left(\frac{I^{[s]}}{I^{n+s}}\right) = \ell_R\left(\frac{R}{I^{n+s}}\right) - \ell_R\left(\frac{R}{I^{[s]}}\right) = H(n+s) - s^d e_0(I). \]
}
\end{proof}
{
\begin{Remarks}{\rm 
	\begin{enumerate}[{\rm (1)}]
		\item The arguments in proof of Theorem \ref{F(n)}, case 3 actually prove that $I^{[s]}I^n = I^{n+s}$, for all $n \geq d(s-1)-s+1.$ In other words, $F(s,n) = H(n+s)-s^de_0(I)$ for all $n \geq s(d-1)-d+1.$ This proves that
		\begin{align} \label{rexpressF(n)}
		F(s,n) = 
		\begin{cases}
		d \ H(n) &  \text{if } 1 \leq n \leq s, \\
		\sum_{i=1}^{d-1} (-1)^{i+1} \binom{d}{i} H(n-(i-1)s) & \text{ if }  s+1 \leq n \leq s(d-1)-d, \\
		H(n+s) - s^d e_0(I) & \text{ if } n \geq s(d-1)-d+1.
		\end{cases}
		\end{align}
		The following claim proves that the above observation does not contradict \eqref{expressF(n)}.\\
		{\bf Claim}: For all $s(d-1)-d+1 \leq n \leq s(d-1),$
		\[ \sum_{i=1}^{d-1} (-1)^{i+1} \binom{d}{i} H(n-(i-1)s) = H(n+s) - s^d e_0(I). \]
		Write $n = s(d-1)-d+1+j,$ where $j=0,1,\ldots,d-1.$ Then
		\begin{align*}
		\sum_{i=1}^{d-1} (-1)^{i+1} \binom{d}{i} H(n-(i-1)s)
		&= \sum_{i=1}^{d-1} (-1)^{i+1} \binom{d}{i} \binom{s(d-1)-d+1+j-(i-1)s+d-1}{d}e_0(I) \\
		&= \sum_{i=1}^{d-1} (-1)^{i+1} \binom{d}{i} \binom{s(d-i)+j}{d}e_0(I).
		\end{align*}
		Using change of variables, we get
		\begin{align*}
		\sum_{i=1}^{d-1} (-1)^{i+1} \binom{d}{i} H(n-(i-1)s)
		&= \sum_{i=1}^{d-1} (-1)^{d-i+1}  \binom{d}{d-i} \binom{is+j}{d} e_0(I) \\
		&= \sum_{i=0}^{d} (-1)^{d-i+1}  \binom{d}{i} \binom{is+j}{d} e_0(I) + \binom{ds+j}{d} e_0(I).
		\end{align*}
		Since $H(n+s) = \binom{ds+j}{d}e_0(I)$, from \cite[3.150]{combi}, we get
		\begin{align*}
		\sum_{i=1}^{d-1} (-1)^{i+1} \binom{d}{i} H(n-(i-1)s)
		&= H(n+s) - s^d e_0(I).
		\end{align*}
		\item Observe that $F(s,n)$ is the Hilbert-Samuel function of the $R$-module $I^{[s]}$ with respect to $I.$ Let $n(I)$ denote the postulation number of $I$ in $R$ and $P_I(n)$ denote the Hilbert-Samuel polynomial of $I$ in $R.$ Using \eqref{rexpressF(n)}, it follows that $F(s,n)$ is a polynomial, given by $P_I(n+s)-s^de_0(I),$ for $n>\max\{ n(I)+s+1, (d-1)s-d+1 \}.$
	\end{enumerate}
}\end{Remarks}
}

We recall the definition of Stirling numbers of the first kind and Stirling numbers of the second kind.
{
\begin{Definition} [{\cite[page 26]{stanley97}}] \label{stirling1}
	{\rm Let $\mathfrak{S}_n$ denote the set of permutations on $n$ elements. Define $c(n,k)$ to be the number of permutations $w \in \mathfrak{S}_n$ with exactly $k$ cycles. The number ${\bf s}(n,k) :=(-1)^{n-k} c(n,k)$ is known as the Stirling number of the first kind. It is easy to check that \\
	(a) ${\bf s}(n,n) = 1,$ \hskip10pt  and  \hskip10pt 
	(b) ${\bf s}(n,n-1) = -\binom{n}{2}.$ \\
	Using \cite[1.3.7 Proposition]{stanley97}, it follows that 
	\begin{align} \label{s1}
	\sum_{k=0}^{n} {\bf s}(n,k) x^k = n! \binom{x}{n}.
	\end{align}
}\end{Definition}
}
\begin{Definition}  [{\cite[page 73]{stanley97}}] \label{stirling2}
	{\rm The Stirling number of the second kind, denoted by ${\bf S}(n,k)$, is equal to the number of partitions of the set $[n] =\{1,\ldots,n\}$ into $k$ blocks.
	\begin{align*}
	{\bf S}(n,k) = \frac{1}{k!} \sum_{i=0}^{k} (-1)^{k-i} \binom{k}{i} i^n.
	\end{align*}
	It is easy to check that \\
	(a) ${\bf S}(d+1,d) = \binom{d+1}{2},$ \hskip40pt
	(b) ${\bf S}(d,d) = 1,$ \hskip35pt  and \hskip40pt
	(c) ${\bf S}(j,d) = 0,$ for all $j<d.$
}\end{Definition}

{
Using Stirling numbers we prove the following combinatorial identity.
\begin{Lemma}  \label{combi}
	Let $s,d \in \mathbb{N}.$ Then
	\begin{align*}
	\sum_{i=0}^{d} (-1)^{d-i} \binom{d}{i} \binom{is}{d+1} = \frac{ds^d(s-1)}{2}.
	\end{align*}
\end{Lemma}
\begin{proof}
	Using \eqref{s1}, write
	\[ \binom{is}{d+1} = \frac{1}{(d+1)!} \sum_{j=0}^{d+1} {\bf s}(d+1,j) (is)^j. \]
	Therefore,
	\begin{align*}
	\sum_{i=0}^{d} (-1)^{d-i} \binom{d}{i} \binom{is}{d+1}
	&= \sum_{i=0}^{d} (-1)^{d-i} \binom{d}{i} \frac{1}{(d+1)!} \sum_{j=0}^{d+1} {\bf s}(d+1,j) (is)^j  \\
	&= \frac{1}{(d+1)!} \sum_{j=0}^{d+1} {\bf s}(d+1,j) \ s^j \ \sum_{i=0}^{d} (-1)^{d-i} \binom{d}{i} i^j.  
	\end{align*}
	Using definition and properties of Stirling numbers (see Definition \ref{stirling1} and \ref{stirling2}), we get
	\begin{align*}
	\sum_{i=0}^{d} (-1)^{d-i} \binom{d}{i} \binom{is}{d+1}
	&= \frac{d!}{(d+1)!} \sum_{j=0}^{d+1} {\bf s}(d+1,j) \ s^j \ {\bf S}(j,d)  \\
	&= \frac{1}{(d+1)} \left[ {\bf s}(d+1,d) \ s^d \ {\bf S}(d,d) + {\bf s}(d+1,d+1) \ s^{d+1} \ {\bf S}(d+1,d) \right]  \\
	&= \frac{1}{(d+1)} \left[ -s^d \binom{d+1}{2} + s^{d+1} \binom{d+1}{2} \right]  \\
	&= \frac{ds^d(s-1)}{2}.
	\end{align*}
\end{proof}
}

\begin{Lemma}  \label{binomial} 
	For $d\geq 2,$ put  $\beta_1 = (d-2)/2$ and $\beta_2 = (d-1)(3d-10)/24$. Then for any  $s \in \mathbb{N},$ 
	\begin{align*}
	\binom{s+d-1}{d+1} = \frac{s^{d+1}}{(d+1)!} + \frac{s^d \beta_1}{d!} + \frac{s^{d-1} \beta_2}{(d-1)!} + \cdots,
	\end{align*}
\end{Lemma}
{
\begin{proof}
	Using \cite[page iv]{combi} and \eqref{s1}, it follows that
	\begin{align*}
	\binom{s+d-1}{d+1} 
	= (-1)^{d+1} \binom{-(s-1)}{d+1}  
	&= \frac{(-1)^{d+1}}{(d+1)!} \sum_{k=0}^{d+1} {\bf s}(d+1,k) (1-s)^k  \\
	&= \frac{(-1)^{d+1}}{(d+1)!} \sum_{k=0}^{d+1} {\bf s}(d+1,k) \sum_{j=0}^{k} \binom{k}{j} (-1)^j s^j.
	\end{align*}
	By index shifting, we get
	\begin{align} \label{b1}
	\binom{s+d-1}{d+1} 
	&= \frac{(-1)^{d+1}}{(d+1)!} \sum_{j=0}^{d+1} \sum_{k=0}^{j} (-1)^{d+1-j} {\bf s}(d+1,d+1-k) \binom{d+1-k}{j-k} s^{d+1-j}.
	\end{align}
	Coefficient of $s^{d+1}$ in \eqref{b1} is
	\begin{align}
	\frac{1}{(d+1)!} \ {\bf s}(d+1,d+1) = \frac{1}{(d+1)!}.
	\end{align}
	Coefficient of $s^d$ in \eqref{b1} is
	\begin{align*}
	\frac{(-1)^{d+1}}{(d+1)!} \left[ (-1)^d {\bf s}(d+1,d+1) (d+1) + (-1)^d {\bf s}(d+1,d) \right]
	&= \frac{-1}{(d+1)!} \left[ d + 1 - \binom{d+1}{2} \right] \\
	&= \frac{d-2}{2d!}. 
	\end{align*}
	Coefficient of $s^{d-1}$ in \eqref{b1} is
	\begin{align*}
	&\frac{1}{(d+1)!} \left[ {\bf s}(d+1,d+1) \binom{d+1}{2} + {\bf s}(d+1,d) d + {\bf s}(d+1,d-1) \right]  \\
	&= \frac{1}{(d+1)!} \left[ \binom{d+1}{2} - d\binom{d+1}{2} + c(d+1,d-1) \right].
	\end{align*}
	Using \cite[1.3.6 Lemma]{stanley97}, it is easy to check that 
	\[ c(d+1,d-1) = \sum_{j=2}^d j \binom{j}{2} = \frac{d(d+1)(3d^2-d-2)}{24}. \]
	Substituting, we get the coefficient of $s^{d-1}$ is
	\begin{align*}
	\frac{1}{(d+1)!} \left[ \binom{d+1}{2} - d\binom{d+1}{2} + \frac{d(d+1)(3d^2-d-2)}{24} \right] = \frac{(d-1)(3d-10)}{24 (d-1)!}.
	\end{align*}
\end{proof}
}

\begin{Theorem} \label{dimdCMsop}
	Let $R$ be a $d$-dimensional Cohen-Macaulay local ring and $d \geq 2.$ Let $I$ be a parameter ideal of $R$ and $\mathcal I=(I,It)\R(I).$ 
	{Put $c(d) = (d/2)+ d/(d+1)!.$} Let $s \in \mathbb{N}.$  \\
	{\rm(1)} Let $s<d.$ Write $d=k_1s+k_2$  where $k_2 \in \{0,1,\ldots,s-1\}.$ If $k_2=0$, then
	\[ \ell_R\left(\frac{\mathcal{R}(I)}{\mathcal I^{[s]}}\right) = 
	e_0(I) \left[ (d-k_1+1)s^{d+1} + d \binom{s+d-1}{d+1} - \sum_{i=0}^{d-1} \left[ (-1)^{i} \binom{d}{i} \binom{(d-i-k_1+1)s+d-1}{d+1} \right] \right]. \]
	If $k_2 \neq 0,$ then
	\[ \ell_R\left(\frac{\mathcal{R}(I)}{\mathcal I^{[s]}}\right) = 
	e_0(I) \left[ (d-k_1)s^{d+1} + d \binom{s+d-1}{d+1} - \sum_{i=0}^{d-1} \left[ (-1)^{i} \binom{d}{i} \binom{(d-i-k_1)s+d-1}{d+1} \right] \right]. \]
	{\rm(2)} Let $s \geq d.$ Then
	{
	\[ \ell_R\left(\frac{\mathcal{R}(I)}{\mathcal I^{[s]}}\right) =  
	e_0(I) \left[ \frac{ds^{d+1}}{2} - \frac{s^d(d-2)}{2} + d \binom{s+d-1}{d+1} \right]. \]  }
	In other words, for $s$ large, 
	\[ \ell_R\left(\frac{\R(I)}{\mathcal I^{[s]}}\right) = c(d)e_0(I) s^{d+1} + e_0(I)\left(\frac{d-2}{2}\right)\left(\frac{1}{(d-1)!} - 1 \right)s^d + e_0(I) \frac{d(d-1)(3d-10)}{24(d-1)!} s^{d-1} + \cdots , \]
	implying that the generalized Hilbert-Kunz multiplicity $e_{HK}((I,It)\R(I)) = c(d) \ e_0(I).$
\end{Theorem}

\begin{proof}
	Consider
	\begin{align*}
	\mathcal I^{[s]} 
	= (I^{[s]}, I^{[s]}t^s) 
	= \left(\bigoplus_{n=0}^{s-1} I^{[s]}I^nt^n \right) + \left(\bigoplus_{n \geq s}  I^{[s]}I^{n-s} t^n \right). 
	\end{align*}
Since $G(I)$ is Cohen-Macaulay and the $a$-invariant $a_d(G(I))=-d<0$, using Theorem \ref{Marley} it follows that $r(I^s)=d-1$ if $s \geq d.$ Let $s<d.$ Write $d=k_1s+k_2$, where $k_2 \in \{0,1,\ldots,s-1\}.$ Then  using Theorem \ref{Marley} it follows that 
	\begin{align*}
		r(I^s) = \lfloor -k_1 - \frac{k_2}{s} \rfloor + d =
		\begin{cases}
			d-k_1 & \text{ if } k_2=0, \\
			d-k_1-1 & \text{ if } k_2 \neq 0. 
		\end{cases}
	\end{align*}
	{\bf Case 1}: Let $s<d.$ Write $d=k_1s+k_2$, where $k_2 \in \{0,1,\ldots,s-1\}.$ Then as observed above, $r(I^s)=d-k_1-j,$ $j \in \{0,1\}.$ As $I^{[s]}$ is a minimal reduction of $I^s$, we get, $I^{[s]}I^{(d-k_1-j)s} = I^{(d-k_1-j+1)s}.$ In other words, $I^{[s]} I^{n-s} = I^n$, for all $n \geq (d-k_1-j+1)s.$ Therefore,
	\begin{align*}
	\mathcal I^{[s]} 
	= \left(\bigoplus_{n=0}^{s-1} I^{[s]}I^nt^n \right) + \left(\bigoplus_{n=s}^{(d-k_1-j+1)s-1} I^{[s]}I^{n-s} t^n \right) + \left(\bigoplus_{n \geq (d-k_1-j+1)s} I^nt^n \right). 
	\end{align*}
	Consider
	\begin{align*}
	\ell_R\left(\frac{\mathcal{R}(I)}{\mathcal I^{[s]}}\right)
	&= \sum_{n=0}^{s-1} \ell_R\left(\frac{I^n}{I^{[s]}I^n}\right) 
	+ \sum_{n=s}^{(d-k_1-j+1)s-1} \ell_R\left(\frac{I^n}{I^{[s]}I^{n-s}} \right) \\
	&= \sum_{n=0}^{s-1} \ell_R\left(\frac{R}{I^{[s]}I^n}\right) 
	+ \sum_{n=s}^{(d-k_1-j+1)s-1} \ell_R\left(\frac{R}{I^{[s]}I^{n-s}} \right) 
	- \sum_{n=0}^{(d-k_1-j+1)s-1} \ell_R\left(\frac{R}{I^n} \right).
	\end{align*}
	{Write $\ell_R \left( \frac{R}{I^{[s]}I^n} \right) = \ell_R \left( \frac{R}{I^{[s]}} \right) + \ell_R \left( \frac{I^{[s]}}{I^{[s]}I^n} \right).$ Since $R$ is Cohen-Macaulay and $I$ is a parameter ideal, $\ell_R \left( \frac{R}{I^n} \right) = e_0(I) \binom{n+d-1}{d}$ for all $n \geq 1$, and hence}
	\begin{align*}
	\ell_R\left(\frac{\mathcal{R}(I)}{\mathcal I^{[s]}}\right)
	&{= (d-k_1-j+1)s \cdot \ell_R\left(\frac{R}{I^{[s]}}\right) 
	+ \sum_{n=0}^{s-1} \ell_R\left(\frac{I^{[s]}}{I^{[s]}I^n}\right) 
	+ \sum_{n=s}^{(d-k_1-j+1)s-1} \ell_R\left(\frac{I^{[s]}}{I^{[s]}I^{n-s}}\right)} \\
	&\hskip7cm {- \sum_{n=1}^{(d-k_1-j+1)s-1} e_0(I)\binom{n+d-1}{d} }\\
	&= (d-k_1-j+1)s \cdot \ell_R\left(\frac{R}{I^{[s]}}\right) 
	+ 2 \sum_{n=0}^{s-1} \ell_R\left(\frac{I^{[s]}}{I^{[s]}I^n}\right) 
	+ \sum_{n=s}^{(d-k_1-j)s-1} \ell_R\left(\frac{I^{[s]}}{I^{[s]}I^n}\right) \\
	&\hskip7cm - \sum_{n=1}^{(d-k_1-j+1)s-1} e_0(I)\binom{n+d-1}{d}.
	\end{align*}
	As $s<d,$ we get $k_1\geq 1.$ Thus, $d-k_1-j \leq d-1.$ Using Theorem \ref{F(n)}, we get
	{\small \begin{align}
	&\ell_R\left(\frac{\mathcal{R}(I)}{\mathcal I^{[s]}}\right) \nonumber\\
	&= (d-k_1-j+1)s^{d+1} e_0(I) + 2d \ e_0(I) \sum_{n=1}^{s-1} \binom{n+d-1}{d} +  d \ e_0(I) \sum_{n=s}^{(d-k_1-j)s-1} \binom{n+d-1}{d} \nonumber\\
	&\hskip2cm + \sum_{n=s}^{(d-k_1-j)s-1} \left[ \sum_{i=2}^{d-1} (-1)^{i+1} \binom{d}{i} e_0(I) \binom{n-(i-1)s+d-1}{d} \right] - \sum_{n=1}^{(d-k_1-j+1)s-1} e_0(I) \binom{n+d-1}{d} \nonumber\\
	&= (d-k_1-j+1)s^{d+1}e_0(I) + 2d \ e_0(I) \binom{s+d-1}{d+1} + d \ e_0(I) \left[ \binom{(d-k_1-j)s+d-1}{d+1} - \binom{s+d-1}{d+1} \right] \nonumber\\
	&\hskip2.5cm - \sum_{i=2}^{d-1} \left[ (-1)^{i} \binom{d}{i} e_0(I) \binom{(d-j-k_1-i+1)s+d-1}{d+1} \right]
	- e_0(I) \binom{(d-k_1-j+1)s+d-1}{d+1} \nonumber\\
	&= (d-k_1-j+1)s^{d+1}e_0(I) + d \ e_0(I) \binom{s+d-1}{d+1} - \sum_{i=0}^{d-1} \left[ (-1)^{i} \binom{d}{i} e_0(I) \binom{(d-j-k_1-i+1)s+d-1}{d+1} \right]. \nonumber
	\end{align}}
	
	{\bf Case 2}: Let $s \geq d.$
	Then $r(I^s)=d-1.$ As $I^{[s]}$ is a minimal reduction of $I^s$, we get, $I^{[s]}I^{(d-1)s} = I^{ds}.$ In other words, $I^{[s]} I^{n-s} = I^n$, for all $n \geq ds.$ Therefore,
	\begin{align*}
	\mathcal I^{[s]} 
	= \left(\bigoplus_{n=0}^{s-1} I^{[s]}I^nt^n \right) + \left(\bigoplus_{n=s}^{ds-1} I^{[s]}I^{n-s} t^n \right) + \left(\bigoplus_{n \geq ds} I^nt^n \right). 
	\end{align*}
	Consider
	\begin{align*}
	\ell_R\left(\frac{\mathcal{R}(I)}{\mathcal I^{[s]}}\right)
	&= \sum_{n=0}^{s-1} \ell_R\left(\frac{I^n}{I^{[s]}I^n}\right) 
	+ \sum_{n=s}^{ds-1} \ell_R\left(\frac{I^n}{I^{[s]}I^{n-s}} \right) \\
	&= \sum_{n=0}^{s-1} \ell_R\left(\frac{R}{I^{[s]}I^n}\right) 
	+ \sum_{n=s}^{ds-1} \ell_R\left(\frac{R}{I^{[s]}I^{n-s}} \right) 
	- \sum_{n=0}^{ds-1} \ell_R\left(\frac{R}{I^n} \right).
	\end{align*}
	{Write $\ell_R \left( \frac{R}{I^{[s]}I^n} \right) = \ell_R \left( \frac{R}{I^{[s]}} \right) + \ell_R \left( \frac{I^{[s]}}{I^{[s]}I^n} \right).$ Since $\ell_R \left( \frac{R}{I^n} \right) = e_0(I) \binom{n+d-1}{d}$, we have}
	\begin{align*}
	\ell_R\left(\frac{\mathcal{R}(I)}{\mathcal I^{[s]}}\right)
	&{= ds \ \ell_R\left(\frac{R}{I^{[s]}}\right)  
	+ \sum_{n=0}^{s-1} \ell_R\left(\frac{I^{[s]}}{I^{[s]}I^n}\right) 
	+ \sum_{n=s}^{ds-1} \ell_R\left(\frac{I^{[s]}}{I^{[s]}I^{n-s}} \right) 
	- \sum_{n=1}^{ds-1} e_0(I)\binom{n+d-1}{d} } \\
	&= ds \ \ell_R\left(\frac{R}{I^{[s]}}\right) 
	+ 2 \sum_{n=0}^{s-1} \ell_R\left(\frac{I^{[s]}}{I^{[s]}I^n}\right) 
	+ \sum_{n=s}^{(d-1)s-1} \ell_R\left(\frac{I^{[s]}}{I^{[s]}I^n}\right) 
	- \sum_{n=1}^{ds-1} e_0(I)\binom{n+d-1}{d}  \\
	&= ds^{d+1}e_0(I) + 2 \sum_{n=1}^{s-1} \ell_R\left(\frac{I^{[s]}}{I^{[s]}I^n}\right) 
	{+ \sum_{n=s}^{(d-1)s-d} \ell_R\left(\frac{I^{[s]}}{I^{[s]}I^n}\right) 
	+ \sum_{n=(d-1)s-d+1}^{(d-1)s-1} \ell_R\left(\frac{I^{[s]}}{I^{[s]}I^n}\right)} \\
	&\hskip9cm - \sum_{n=1}^{ds-1} e_0(I)\binom{n+d-1}{d}.
	\end{align*}
	From \eqref{rexpressF(n)}, it follows that
	{
	\begin{align*} 
	\ell_R\left(\frac{\mathcal{R}(I)}{\mathcal I^{[s]}}\right)
	&= ds^{d+1} e_0(I) + 2d \ e_0(I) \sum_{n=1}^{s-1} \binom{n+d-1}{d} +  d \ e_0(I) \sum_{n=s}^{(d-1)s-d} \binom{n+d-1}{d}  \\
	&+ \sum_{n=s}^{(d-1)s-d} \left[ \sum_{i=2}^{d-1} (-1)^{i+1} \binom{d}{i} e_0(I) \binom{n-(i-1)s+d-1}{d} \right]  \\
	&+\sum_{n=s(d-1)-d+1}^{(d-1)s-1} \left[ \binom{n+s+d-1}{d} e_0(I) - s^de_0(I) \right]
	- \sum_{n=1}^{ds-1} e_0(I) \binom{n+d-1}{d}  \\
	&= ds^{d+1}e_0(I) + 2d \ e_0(I) \binom{s+d-1}{d+1} + d \ e_0(I) \left[ \binom{(d-1)s}{d+1} - \binom{s+d-1}{d+1} \right]   \\
	&- \sum_{i=2}^{d-1} \left[ (-1)^{i} \binom{d}{i} e_0(I) \binom{(d-i)s}{d+1} \right]
	+ e_0(I)\binom{ds+d-1}{d+1} - e_0(I)\binom{ds}{d+1} - s^d(d-1)e_0(I)   \\
	&\hskip9cm- e_0(I) \binom{ds+d-1}{d+1}  \\
	&= ds^{d+1}e_0(I) + d \ e_0(I) \binom{s+d-1}{d+1} - \sum_{i=0}^{d-1} \left[ (-1)^{i} \binom{d}{i} e_0(I) \binom{(d-i)s}{d+1} \right] - s^d(d-1)e_0(I)  \\
	&= e_0(I) \left[ds^{d+1} - s^d(d-1) + d \binom{s+d-1}{d+1} - \sum_{i=1}^{d} \left[ (-1)^{d-i} \binom{d}{i} \binom{is}{d+1} \right] \right].  
	\end{align*}}

	{
	Using Lemma \ref{combi}, it follows that
	\begin{align} \label{HKfunction}
	\ell_R\left(\frac{\mathcal{R}(I)}{\mathcal I^{[s]}}\right)
	&= e_0(I) \left[ds^{d+1} - s^d(d-1) + d \binom{s+d-1}{d+1} - \frac{ds^d(s-1)}{2} \right]  \nonumber\\
	&= e_0(I) \left[ \frac{ds^{d+1}}{2} - \frac{s^d(d-2)}{2} + d \binom{s+d-1}{d+1} \right].
	\end{align}
	}
	This implies that the Hilbert-Kunz function is a polynomial for $s \geq d.$ We now use Lemma \ref{binomial} to find the coefficients of $s^{d+1}$, $s^d$ and $s^{d-1}$ in the expression \eqref{HKfunction}. The coefficient of $s^{d+1}$ is
	\begin{align*}
	e_0(I) \left[ \frac{d}{2} + \frac{d}{(d+1)!} \right] 
	= c(d).
	\end{align*}
	Coefficient of $s^d$ is
	\begin{align*}
	e_0(I) \left[ \frac{-(d-2)}{2} + \frac{d(d-2)}{2d!} \right] 
	&= e_0(I) \left( \frac{d-2}{2} \right) \left[ -1 + \frac{1}{(d-1)!} \right]
	\end{align*}
	and coefficient of $s^{d-1}$ is 
	\begin{align*}
	e_0(I) \frac{d(d-1)(3d-10)}{24(d-1)!}.
	\end{align*}
\end{proof}

{In \cite{conca}, Conca remarked that the generalized Hilbert-Kunz function is dependent on the choice of the set of generators of the ideal. But if $R$ is a Cohen-Macaulay local ring of dimension $d \geq 2$ and $I$ is a parameter ideal, then using Theorem \ref{dimdCMsop} it follows that the length function $\ell(\R(I)/(I,It)^{[s]})$ is independent of the choice of generators of $I.$}

\begin{Corollary}
	Let $(R,\mm)$ be a $d$-dimensional regular local ring with $d\geq 2.$ Then for $s \geq d$,
	\[ \ell_R\left(\frac{\R(\mm)}{(\mm,\mm t)^{[s]}}\right) = c(d) s^{d+1} + \left(\frac{d-2}{2}\right)\left(\frac{1}{(d-1)!} - 1 \right)s^d + \frac{d(d-1)(3d-10)}{24(d-1)!} s^{d-1} + \cdots . \]
\end{Corollary}

\begin{proof}
	Use $I=\mm$ in Theorem \ref{dimdCMsop}.
\end{proof}

\begin{Corollary} \label{I=m}
	Let $(R,\mm)$ be a Cohen-Macaulay local ring of dimension $d$ with positive prime characteristic $p>0.$ If there exists a parameter ideal $I$ of $R$ such that $I^*=\mm,$ then $$e_{HK}(\R(\mm))=c(d)e_0(\mm).$$
\end{Corollary}

\begin{proof}
	If there exists a parameter ideal $I$ of $R$ such that $I^*=\mm,$ then using the same arguments as in the proof of \cite[Corollary 4.5(2)]{etoYoshida}, it follows that 
	$e_{HK}((\mm,\mm t)\R(\mm)) = e_{HK}((I,It)\R(I)).$ Therefore, 
	\[e_{HK}(\R(\mm))=e_{HK}((\mm,\mm t)\R(\mm)) = e_{HK}((I,It)\R(I)) = c(d)e_0(I) = c(d)e_0(\mm).\] The latter equality holds as $I^*=\mm$ implies that $\overline{I}=\mm.$
\end{proof}

\begin{Corollary}
	Let $R$ be a $d$-dimensional Cohen-Macaulay Stanley-Reisner ring of a simplicial complex over an infinite field with prime characteristic $p>0.$ Let $\mm$ be a maximal homogeneous ideal of $R.$ Then $e_{HK}(\R(\mm))=c(d) \ f_{d-1}$, where $f_{d-1}$ is the number of facets in the simplicial complex.
\end{Corollary}

\begin{proof}
	Let $I$ be an ideal generated by a linear system of parameters. It is proved in \cite[Theorem 6.1]{gmv} that $\mm=I^*.$ Therefore, using Corollary \ref{I=m} it follows that $e_{HK}(\R(\mm))=c(d) \ e(\mm).$ Since $e(\mm)=f_{d-1}$, we are done.
\end{proof}

\begin{Example}
	Let $k[[X_1,X_2,X_3]]$ be a power series ring in $3$ variables over a field $k.$ Let $I=(X_1^{n_1}, X_2^{n_2}, X_3^{n_3})$, where $n_1,n_2,n_3 \in \NN.$ Then
	\begin{align*} 
	\ell_R\left(\frac{\mathcal{R}(I)}{(I,It)^{[2]}}\right) 
	= e_0(I) \left[ 32 + 3 - \sum_{i=0}^{2} \left[ (-1)^{i} \binom{3}{i} \binom{6-2i}{4} \right] \right] = 23 \ e_0(I) = 23 \ n_1n_2n_3
	\end{align*}
	and for $s \geq 3,$
	\begin{align*}
	\ell_R\left(\frac{\mathcal{R}(I)}{\mathcal I^{[s]}}\right)
	= n_1n_2n_3 \left[ \frac{3s^{4}}{2} - \frac{s^3}{2} + 3 \binom{s+2}{4} \right] 
	= n_1n_2n_3 \left[ \frac{13}{8} s^4 - \frac{1}{4} s^3 - \frac{1}{8} s^2 - \frac{1}{4} s \right].
	\end{align*}
	In particular, for $s \geq 3$,
	\begin{align*} 
	\ell_R\left(\frac{\mathcal{R}(\mm)}{(\mm,\mm t)^{[s]}}\right) 
	= \frac{13}{8} s^4 - \frac{1}{4} s^3 - \frac{1}{8} s^2 - \frac{1}{4} s.
	\end{align*} 
\end{Example}

\begin{Example}
	Let $R=k[[X,Y,Z]]/(XY-Z^n)$ for some positive integer $n \geq 2$ and let $I$ be a parameter ideal of $R.$ Then for $s \geq 2,$
	\begin{align*}
	\ell_R\left(\frac{\mathcal{R}(I)}{\mathcal I^{[s]}}\right) 
	= e_0(I) \left[ s^{3} + 2 \binom{s+1}{3} \right] 
	= e_0(I) \left[ \frac{4}{3} s^3 - \frac{1}{3} s \right].
	\end{align*}
\end{Example}

\bibliographystyle{plain}
\bibliography{Notes}

\end{document}